\documentclass{svmult}
\usepackage{amssymb,amsmath,amscd}
\usepackage[T1]{fontenc}

\usepackage{mathrsfs}
\usepackage{color}
\usepackage{epsfig}

\usepackage{stmaryrd} 

\usepackage[all]{xy}

\renewenvironment{quote}{\list{}{\leftmargin=0.7in\rightmargin=0.7in}\item[]}{\endlist}

\newcommand{\ZZ}{\mathbb{Z}}
\newcommand{\QQ}{\mathbb{Q}}

\newcommand{\CC}{\mathbb{C}}

\newcommand{\PP}{\mathbb{P}}

\newcommand{\cM}{\mathcal{M}}
\newcommand{\cMb}{\overline{\mathcal{M}}}
\newcommand{\cO}{\mathcal{O}}

\newcommand{\idx}[1]{\index{#1}{\em #1}}

\DeclareMathOperator{\vdim}{vdim}

\makeindex

\begin{document}

\title{Introduction to Gromov-Witten theory}

\author{Simon C. F. Rose}
\institute{Field's Institute}

\maketitle

\abstract{The goal of these notes is to provide an informal introduction to Gromov-Witten theory with an emphasis on its role in counting curves in surfaces. These notes are based on a talk given at the Fields Institute during a week-long conference aimed at introducing graduate students to the subject which took place during the thematic program on Calabi-Yau Varieties: Arithmetic, Geometry, and Physics.}

\section{Introduction}

We begin with a natural, and quite old, question.

\begin{question}
How do you count rational curves in a smooth variety?
\end{question}

\begin{example}
The simplest examples of this are the following.
\begin{enumerate}
\item How many straight lines pass between two points in the plane?

This question is one of my favourites to ask people when they ask me to explain what I do as a mathematician. The answer is very easy, but most people who hear the question expect it to be a trick question, not realizing that the answer is the relatively obvious answer of 1.
\item How many conics pass through five points in the plane?

The fact that this is 1 is a classically known fact. How can you show this? One of my favourite was of doing so is to construct the solution explicitly, which we can do as follows.

Choose 5 points \((x_i, y_i)\) in general position, and consider the following determinant.
\[
\begin{vmatrix}
1 & X   & Y   & X^2   & Y^2   & XY \\
1 & x_1 & y_1 & x_1^2 & y_1^2 & x_1y_1 \\
1 & x_2 & y_2 & x_2^2 & y_2^2 & x_2y_2 \\
  &     &     & \vdots \\
1 & x_5 & y_5 & x_5^2 & y_5^2 & x_5y_5 \\
\end{vmatrix}
\]
This gives you a polynomial \(f(X,Y)\) of degree (at most) 2. With a little bit of thought, you should be able to see that this polynomial is exactly the conic that passes through these five points.

\item How many (nodal) cubics pass through eight points in the plane?

The answer to this, 12, is also classically known, but is more subtle. It seems na\"ively that one should be able to use the exact same trick as before to determine a cubic which passes through our 8 points. However, as we shall see, it would take 9 points to do this, and in the end we would find a smooth genus 1 curve, not a rational curve. Consequently, the method that we use for conics does not work here.
\end{enumerate}
\end{example}

It is not too hard to see the pattern here. We are interested in counting degree \(d\) curves in the plane (which for technical reasons we consider to be \(\PP^2\) instead of \(\CC^2\)), and we see that we are imposing the condition that it pass through \(3d - 1\)  points (in general position). For example, \(d = 1\) is a line, \(d = 2\) is a conic, etc.

This of course raises a natural question. {\em Why} is it \(3d - 1\) points?

\begin{lemma}
The space of nodal, rational degree \(d\) curves is \((3d - 1)\)-dimensional.
\end{lemma}

\begin{proof}
Despite the earlier comment, we will work in the affine setting for simplicity. In such a case, a degree \(d\) curve is given by the zeros of a polynomial
\[
f(x, y) = \sum_{0 \leq i + j \leq d} a_{i,j}x^iy^j
\]
with at least one of the terms \(a_{i,j}\) with \(i + j = d\) is not zero.

The space of all such polynomials is given by varying the coefficients. It is easy to see that there are
\[
\sum_{0 \leq i + j \leq d} 1 = {d + 2 \choose 2}
\]
of these, and so the space of degree \(d\) plane curves has dimension \({d + 2 \choose 2} - 1\), since we only care about the zeros of the polynomial (i.e. we only care about the polynomial up to an overall scaling factor).

Now, a generic curve in such a family will be smooth and have genus
\[
g(C) = \frac{(d - 1)(d - 2)}{2}
\]
(see exercise \ref{ex_genus_plane_curve}) and so if we impose the condition that the curve have \(g(C)\) nodes (which are each codimension one conditions), then we see that the resulting curve will be rational, and the space will have dimension
\[
\frac{(d + 2)(d + 1)}{2} - 1 - \frac{(d - 1)(d - 2)}{2} = 3d - 1
\]
as claimed.
\end{proof}

We now see why we need \(3d-1\) points, but this then pushes us forward to the next question. How do we actually count the number of these curves?

\section{Moduli of stable maps}

There are a number of different ways of counting curves in a smooth variety \(X\) (see in particular the paper \cite{pt_ways}). In dimension three, following through with the idea above leads to Donaldson-Thomas theory. That is, we can study curves by understanding the local equations which define them---by looking at the sheaves
\[
0 \to \cO_X(-C) \to \cO_X \to \cO_C \to 0
\]
or more specifically, the object \([\cO_X \to \cO_C]\) in the derived category \(D^b(Coh(X))\).

We will take a different approach. One downside of embedded curves \(C \subset X\) is that the singularities of such curves can be arbitrarily bad. However, if we try to understand these curves by parametrizing them---that is, by looking at maps \(f : \Sigma \to X\) such that \(f(\Sigma) = C\)---then we can restrict ourselves to curves \(\Sigma\) with at worst {\em nodal} singularities.

To further describe this, we will need a few definitions.

\begin{definition}
A genus \(g\), \(n\)-marked \idx{pre-stable curve} consists of the data \((C, x_1, \ldots, x_n)\) where
\begin{enumerate}
\item \(C\) is a (possibly nodal) curve of arithmetic genus \(g\) (i.e. \(\chi(\cO_C) = 1 - g\)).
\item \(x_i\) are smooth points of \(C\)
\end{enumerate}
Furthermore, the curve \((C, x_1, \ldots, x_n)\) is {\em stable} if it has only finitely many automorphisms.
\end{definition}

There is a well-defined moduli ``space'' of such curves (actually, a stack, or an orbifold), which we denote by \(\overline{\cM}_{g,n}\), which has dimension \(3g - 3 + n\). This is a classical object of interest, which has been studied in many different ways. For our purposes, we mostly only consider it to be of tangential interest.

Our next definition is the more important one for our purposes.

\begin{definition}
Let \(X\) be a smooth projective variety (this can be relaxed somewhat). Then a genus \(g\), \(n\)-marked \idx{stable map} into \(X\) consists of the following data.
\begin{enumerate}
\item A genus \(g\), pre-stable curve \((C, x_1, \ldots, x_n)\).
\item A map \(f : C \to X\) with only finitely many automorphisms
\end{enumerate}
where an automorphism of a map \(f\) is a map \(h\) for which the following diagram commutes.
\[
\xymatrix{
(C, x_1, \ldots, x_n) \ar[rr]^f\ar[d]_h & & X \\
(C, x_1, \ldots, x_n) \ar[urr]_f
}
\]
\end{definition}

What we want to do is to consider the moduli ``space'' of such objects. That is, we consider
\[
\cMb_{g,n}(X) = \{f : (C, x_1, \ldots, x_n) \to X \mid f \text{ is stable}\},
\]
whatever this object may be\footnote{The best way to define this is as a category whose objects are flat families of stable maps, and whose morphisms are commutative cartesian diagrams. That this is a category is reasonably clear; that it is any sort of ``space'' is far less so. However, many of the other structures described below become fairly clear in this context. For good references (admittedly, in the orbifold setting), see \cite{agv, agv_long, gillam}.}. We further refine this by the discrete data of the homology class of the image of \(f\). That is, fix a homology class \(\beta \in H_2(X)\). We denote by
\[
\cMb_{g,n}(X,\beta) = \{f : (C, x_1, \ldots, x_n) \to X \mid f_*[C] = \beta, f \text{ is stable}\}
\]
We should remark that this is an empty moduli space of the homology class \(\beta\) does not support holomorphic curves.

Let us look at a few examples.

\begin{example}
The simplest example is that of lines in \(\PP^2\). Consider
\[
\cMb_{0,0}(\PP^2, 1)
\]
(where we use the convention that, if \(H_2(X) \cong \ZZ\), then we use an integer to represent the homology class which is that multiple of a generator).

This parameterizes maps \(\PP^1 \to \PP^2\) up to reparameterization of the map. That is, this is nothing but the collection of lines in \(\PP^2\), or \((\PP^2)^*\). This is obviously smooth, compact, and irreducible. It is the best of all worlds.
\end{example}

\begin{example}
Things rapidly degenerate from here, however. Let us consider the next simplest case, that of conics. Consider
\[
\cMb_{0,0}(\PP^2, 2).
\]
This {\em should} be the space of conics in \(\PP^2\), but it is not. This moduli space is built up as follows.

\begin{enumerate}
\item There is an locus of maps whose sources are smooth (this is denoted \(\cM_{0,0}(\PP^2,2)\)---note the lack of a bar over the \(\mathcal{M}\).). Generically, the image of such a map will be a smooth conic, which will be an open locus in the space of all conics, which itself is isomorphic to \(\PP^5\).

\item Within this locus of maps, there is a sublocus consisting of those maps which map as \(2 : 1\) covers of a line in \(\PP^2\). This is a 4-dimensional locus, since we need two parameters to describe the target line, and two to describe the ramification points of the map. Each map in this locus also has \(\ZZ/2\) as an automorphism group, coming from the exchange of the covering sheets.

\item At the `boundary' (i.e. in \(\cMb_{0,0}(\PP^2, 2) \setminus \cM_{0,0}(\PP^2, 2)\)), there are those curves whose domains consist of a nodal curve with two components, each of which maps with degree 1 into \(\PP^2\). Within this, there is the locus of those maps with image two distinct lines (which necessarily join at one point). This is a four dimensional space, two for each line in \(\PP^2\).

\item Deeper into the boundary, there is the locus of those curves with nodal sources, but whose image are both the same line. This is three dimensional; two for the line, and one for the point on that line where the two components meet. Furthermore, every map in this locus also has  \(\ZZ/2\) as automorphism group, since there is an automorphism of the source curve which exchanges the two components.
\end{enumerate}



\end{example}


Despite the fact that there are multiple components of differing dimensions, this is not all that bad. There is a dense open set consisting of the smooth conics, and as we move to higher and higher codimension strata, our curves/maps degenerate in predictable ways.

This is, however, not always the case.

\begin{example}
Let us consider the moduli space \(\cMb_{1,0}(\PP^2, 1)\). This is supposed to be the moduli space of genus 1 maps into \(\PP^2\) of degree 1. It is tempting to say that this is empty (after all, a genus 1 curve in \(\PP^2\) must have degree at least 3). However, this is not the case.

What is true is that the ``open locus'' of smooth curves is empty. That is, if as above we define \(\cM_{g,n}(X, \beta)\) to be the collection of stable maps into \(X\) whose source curve is smooth, then we certainly have
\[
\cM_{1,0}(\PP^2, 1) = \emptyset.
\]
However, there are non-smooth maps. Consider a curve which is a genus 0 curve connected to a genus 1 curve at one point; this is a pre-stable curve. Moreover, we can map it into \(\PP^2\) by mapping the genus 0 curve onto a line, and by collapsing the genus 1 curve to a point. Consequently, this space is 4-dimensional: two for the line, one for the point on the line, and one for the modulus of the elliptic curve.
\end{example}

We can say the following. From the Hirzebruch-Riemann-Roch formula (see \cite{hirz-book} and \cite[Section 7.1.4]{coxetal} for its application in this context), we can say that the \idx{expected dimension} of the moduli space \(\cMb_{g,n}(X,\beta)\) is given by the formula
\[
\vdim \cMb_{g,n}(X, \beta) = (\dim X - 3)(1 - g) + \int_\beta c_1T_X + n
\]
In each of the three examples above, this is respectively 2, 5, and 3. In the first two cases, this is the top dimension of the moduli space, and so all is well. In the case of \(\cMb_{1, 0}(\PP^2, 1)\), we saw that the dimension was 4, while the virtual (expected) dimension is only 3.

Moreover, we can also see that, for rational curves in \(\PP^2\), that this formula is consistent with our \(3d - 1\) points discussion earlier. Since \(K_{\PP^2} \cong \cO(-3)\), it follows that
\[
\int_{dH} c_1T_{\PP^2} = 3d
\]
and hence the virtual dimension of \(\cMb_{0,0}(\PP^2,d) = 3d - 1\).

Lastly, we should note that this also suggests part of our general interest in Calabi-Yau threefolds. In such a case, we see that most of the terms in the dimension formula vanish: \(\dim X = 3\) covers the first term, while \(c_1T_X = 0\) covers the second. Thus, if \(X\) is a Calabi-Yau threefold, then
\[
\vdim \cMb_{g,0}(X, \beta) = 0
\]
and so we should generically expect finitely many curves of any genus in one of these varieties.

Of course, the reality is much more complex.

\section{Gromov-Witten Invariants}

We want to use the previously discussed \(\cMb_{g,n}(X,\beta)\) to count holomorphic curves in \(X\). In an ideal world, such a ``space'' would be both smooth, compact, have components all of the same dimension (the expected dimension, of course), all of which would allow us to use intersection theory to count curves. For the sake of exposition, let us make these simplifying assumptions to see where we can go from this.

We first note that this space comes together with some \idx{evaluation maps} to \(X\). That is, there are maps
\[
ev_i : \cMb_{g,n}(X,\beta) \to X
\]
defined by
\[
\big(f : (C, x_1, \ldots, x_n) \to X\big) \mapsto f(x_i)
\]

Consider now subvarieties \(V_1, \ldots, V_n\) of \(X\). Their homology classes have Poincar\'e duals \(\gamma_i \in H^{n_i}(X)\), and so we can consider the cohomology class \(ev_i^*\gamma_i\). The Poincar\'e dual of this class represents the collection of maps \(f : C \to X\) such that \(f(x_i) \in V_i\). Moreover, since the cup product is Poincar\'e dual to intersection for smooth manifolds, we have that
\[
ev_1^*\gamma_1 \smile \cdots \smile ev_n^*\gamma_n
\]
represents exactly (in a suitably generic setting) those maps \(f : C \to X\) such that \(f(x_i) \in V_i\) for all \(1 \leq i \leq n\). Since the location of the points on \(C\) is arbitrary (i.e. varies over the moduli space), we can read this as
\vskip0.75cm

\begin{quote}
The cohomology class \(ev_1^*\gamma_1 \smile \cdots \smile ev_n^*\gamma_n\) represents the
collection of morphisms \(f : C \to X\) such that the image \(f(C)\) intersects \(V_i\) for all \(1 \leq i \leq n\).
\end{quote}
\vskip0.75cm
If this is a finite number (which should generically occur if this class is a top class in \(H^*\big(\cMb_{g,n}(X,\beta)\big)\)), then by pairing it with the fundamental class we should get the number of such curves. That is, if we consider the integral
\[
\int_{\cMb_{g,n}(X,\beta)}ev_1^*\gamma_1 \smile \cdots \smile ev_n^*\gamma_n
\]
then this number is exactly the number of genus g curves in \(X\) such that they have non-zero intersection with the subvarieties \(V_1, \ldots, V_n\) as desired.

Now, we have assumed for the purposes of this discussion that the moduli space is smooth, compact, and finite-dimensional. Unfortunately, this is not necessarily true. It is proper (compact), but it is often not smooth, and it often has many different components of varying dimensions as we saw before.

The resolution of this is the following. By a general construction due to Behrend-Fantechi (\cite{bf_normal_cone}), we can always construct a so-called \idx{virtual fundamental class} for \(\cMb_{g,n}(X,\beta)\). This is a homology class which we denote as
\[
[\cMb_{g,n}(X,\beta)]^{vir}
\]
which satisfies a number of properties which make it work much like the ordinary fundamental class. The simplest of these is that it is a homology class of pure dimension, which is the expected dimension of \(\cMb_{g,n}(X,\beta)\), which puts us a case that resembles the ideal one described above.

\begin{remark}
We will (as with most introductions to Gromov-Witten theory) largely ignore any technical issues surrounding the virtual fundamental class. For the most part we will treat it as if it is the ordinary fundamental class. For any cases where this is not true, we will do our best to highlight those cases clearly.
\end{remark}

With that in mind, we define the following.

\begin{definition}
Let \(\gamma_1, \ldots, \gamma_n \in H^*(X)\). We define the corresponding \idx{Gromov-Witten invariant} to be
\[
\langle \gamma_1, \ldots, \gamma_n\rangle_{g, \beta}^X = \int_{[\cMb_{g,n}(X,\beta)]^{vir}}ev_1^*\gamma_1 \smile \cdots \smile ev_n^*\gamma_n
\]
where the maps \(ev_i\) are the evaluation maps discussed before.
\end{definition}

\begin{quote}
\begin{remark}
A small public service announcement: For those of you who are reading this who work with \LaTeX, the symbols \(\langle, \rangle\) used above are {\em not} less than/greater than signs (\(<, >\)). These should not be used as delimiters, as the default spacing for them in \LaTeX\ is that of a relation symbol. Furthermore, they just look a little squashed and silly.

Instead, you should use the terms \verb|\langle| and \verb|\rangle| (\(\langle\) and \(\rangle\), respectively). Not only do they look better, but the whitespace around them also looks better, and they can even be resized depending on contents with the use of \verb|\left\langle| and \verb|\right\rangle|:
\[
\left\langle \frac{1}{2}\right\rangle
\]

This has been your public service announcement. We now return to the regular programming.
\end{remark}
\end{quote}

\begin{remark}
It is clear from the fact that \([\cMb_{g,n}(X,\beta)]^{vir}\) is of pure dimension equal to the expected dimension that the Gromov-Witten invariant \(\langle \gamma_1, \ldots, \gamma_n\rangle_{g, \beta}^X\) is zero unless
\[
\sum_{i=1}^n \deg \gamma_i = 2 \vdim \cMb_{g,n}(X, \beta).
\]
\end{remark}

As stated above, in the ideal case the Gromov-Witten invariant \(\langle \gamma_1, \ldots, \gamma_n\rangle_{g, \beta}^X\) provides a count of the number of holomorphic curves in \(X\) that intersect varieties \(V_i\) whose Poincar\'e duals are given by \(\gamma_i\).

\begin{example}
Let \(L\) denote the class of a line in \(\PP^2\), and so \(L^2 = pt\) is the class of a point. Then the Gromov-Witten invariant
\[
\langle \underbrace{pt, \ldots, pt}_{3d-1} \rangle_{0, d}^{\PP^2}
\]
is (in this case) exactly the number of degree \(d\) rational curves in \(\PP^2\) passing through \(3d - 1\) points.
\end{example}

So what does this give us? On the surface, it doesn't simplify matters very much at all. All we have done is replaced difficult enumerative computations with a somewhat abstract and complicated computational formalism. It is not clear that this is of any use in telling us how many degree \(d\) rational curves there are in \(\PP^2\) passing through \(3d - 1\) points.

It turns out, however, that the introduction of all of this structure helps quite a lot. In particular, these moduli spaces have a number of different maps between them which will allow us a lot of leeway to compute Gromov-Witten invariants.

Let us list a few of these maps, and then we will discuss some of their properties.

First of all, we have \idx{forgetful morphisms}. Assuming the latter moduli space exists, these are morally given by morphisms (we will explain the subtlety shortly)
\[
\pi_{n+1} : \cMb_{g,n+1}(X,\beta) \to \cMb_{g,n}(X, \beta)
\]
where we map
\[
\big( f: (C, x_1, \ldots, x_{n+1}) \to X \big) \mapsto \big( f: (C, x_1, \ldots, x_{n}) \to X \big).
\]
That is, we forget the \((n+1)\)-st marked point of the source curve.



We also have a map
\[
\pi : \cMb_{g,n}(X,\beta) \to \cMb_{g,n}
\]
which forgets the map (and the target space), provided again that the latter moduli space exists (i.e. provided that \(2g - 2 + n > 0\). For the most part, we will not use this map, although it can be used to provide a parallel definition of Gromov-Witten invariants. That is, given the diagram
\[
\xymatrix{
\cMb_{g,n}(X, \beta) \ar[d]_\pi \ar[rr]^{ev} & & X \times \cdots \times X \\
\cMb_{g,n}
}
\]
we can equally define the Gromov-Witten invariant
\[
\langle \gamma_1, \ldots, \gamma_n \rangle_{g,\beta}^x = \pi_* ev^*(\gamma_1 \times \cdots \times \gamma_n) \frown [\cMb_{g,n}]
\]
which makes sense as \(\cMb_{g,n}\) is a smooth orbifold, and so has a well-defined fundamental class (in the ordinary sense).

\begin{remark}
We must be slightly careful with both of these forgetful morphisms. Let us focus on the forgetful map \(\cMb_{g,n}(X, \beta) \to \cMb_{g,n}\). Recall that this takes a stable map \((f : C \to X)\) and maps it to the underlying curve. The issue is that the underlying curve itself is only pre-stable, and so may not actually lie in the moduli space \(\cMb_{g,n}\). For example, if we had the dual graph of \(C\) given by
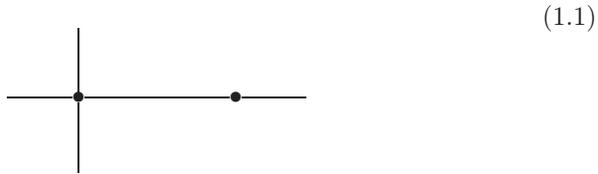
\begin{equation}\label{eq_pre_stable_curve}
\xymatrix{
& &\\
& *{\bullet} \ar@{-}[l]\ar@{-}[u]\ar@{-}[d]\ar@{-}[rr] && *{\bullet} \ar@{-}[r] &\\
&
}
\end{equation}
(with the vertices representing genus 0 irreducible components, and the tails representing marked points), then this is a stable map if it is non-constant on the component with one marked point (the other component has four `special' points). However, the curve itself is not stable, since this other component only has two `special' points.

The solution is to {\em stabilize} the underlying curve. The idea is simply to collapse any components of the curve with too few `special' points. For example, the curve shown in \eqref{eq_pre_stable_curve} would be stabilized to
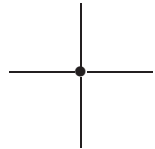
\[
\xymatrix{
& \\
& *{\bullet} \ar@{-}[l]\ar@{-}[r]\ar@{-}[u]\ar@{-}[d] & \\
&
}
\]
which is a genus zero curve with four marked points, as we would expect. The key technicality is that we can do this in families, a technique called {\em stable reduction} (see \cite{knud}). This works similarly with the maps which forget marked points
\end{remark}

Using these maps (and some properties of the virtual fundamental class), we can show that the Gromov-Witten invariants satisfy the following axioms.

\begin{enumerate}
\item Fundamental Class Axiom: We have the equality
\[
\langle \gamma_1, \ldots, \gamma_{n-1}, [X]^\vee\rangle_{g,\beta}^X = \langle \gamma_1, \ldots, \gamma_{n-1}\rangle_{g,\beta}^X
\]

We can think of this as saying that imposing the constraint that a point on our curve be incident to \(X\) is no condition at all. This has the further consequence that
\[
\langle \gamma_1, \ldots, \gamma_{n-1}, [X]^\vee\rangle_{g,\beta}^X = 0
\]
provided that \(n + 2g \geq 0\) or that \(\beta \neq 0\) and \( n \geq 1\). This is since the moduli spaces in question on the left- or right-hand side have different dimension; it thus follows that if the forgetful map exists, then we must have that the Gromov-Witten invariants are zero.

\item Divisor Axiom: If the same conditions are satisfied, and if \(\gamma_n \in H^2(X, \QQ)\), then
\[
\langle \gamma_1, \ldots, \gamma_{n-1},\gamma_n\rangle_{g, \beta}^X = \Big(\int_\beta \gamma_n \Big) \langle \gamma_1, \ldots, \gamma_{n-1}\rangle_{g, \beta}^X
\]

In this case, this is morally due to the fact that the possible number of points that a curve may intersect a divisor is exactly \(\int_\beta \gamma_n\).

\item Point Mapping Axiom: The invariants with \(\beta = 0\), \(\langle \gamma_1, \ldots, \gamma_n \rangle_{g,0}^X\), satisfy
\[
\langle \gamma_1, \ldots, \gamma_n \rangle_{g,0}^X =
\begin{cases}
\int_X \gamma_1 \smile \gamma_2 \smile \gamma_3 & n = 3 \\
0 & \text{otherwise}
\end{cases}
\]

From this we note that the Gromov-Witten invariants of \(X\) include as special cases the triple products in cohomology.
\end{enumerate}

These three axioms together tell us a lot about the Gromov-Witten theory of varieties of dimension 1 and 2. In particular, the divisor axiom in both cases reduces us to computing (for surfaces)
\[
\langle \underbrace{pt, \ldots, pt}_n \rangle_{g,\beta}^S
\]
where
\[
n = (g - 1) + \int_\beta c_1T_S
\]
from which we can compute all other invariants.

For curves, this reduces even further: point insertions are divisors, and so the only invariant we need to compute is the empty bracket. But this will only make sense if the virtual dimension of \(\cMb_{g,0}(C, d[C])\) is zero. That is, if
\[
2g - 2 + d\big(2 - 2g(C)\big) = 0
\]
or equivalently, that \(g = dg(C) - d + 1\). In such a case, we are considering unramified covers of the target \(C\), which can be counted by enumerating index \(d\) subgroups of the fundamental group of \(C\); in particular, if \(g = g(C) = 1\), then we are enumerating index \(d\) sublattices of \(\ZZ^2\), the number of which is classically known (see exercise \ref{ex_lattices}) to be given by \(\sigma_1(d) = \sum_{k \mid d} k\).

It follows then that Gromov-Witten theory lets us count the number of unramified covers of a target curve. What about ramified covers? It turns out that, with some care, we can study this by looking at so-called {\em descendent invariants}. We will, however, omit this discussion from these notes. For a thorough discussion about this matter, see \cite{op_gw_hurwitz}.

\section{Gromov-Witten Potential}

The key to working with Gromov-Witten invariants to their full potential is to do what one should always do when confronted with an infinite collection of numbers depending on discrete data: arrange them into a generating function.

In order to do so, we need to fix some notation. As before, let \(X\) be a smooth projective variety, and let \(\gamma_0, \ldots, \gamma_m\) be a basis of \(H^*(X)\) such that
\begin{enumerate}
\item \(\gamma_0 = 1 = [X]^\vee\in H^0(X)\)
\item \(\gamma_1, \ldots, \gamma_r\) is a basis of \(H^2(X)\).
\end{enumerate}

\begin{definition}
We define the genus \(g\) \idx{Gromov-Witten potential function} of \(X\) to be the formal series 
\[
\Phi_g^X(y_0, \ldots, y_m, q) = \sum_{k_0, \ldots, k_m} \sum_{\beta \in H_2(X)} \langle\gamma_0^{k_0}, \ldots, \gamma_m^{k_m}\rangle_{g,\beta}^X \frac{y_0^{k_0}}{k_0!} \cdots \frac{y_m^{k_m}}{k_m!} q^\beta
\]
\end{definition}

\begin{remark}
The \(q^\beta\) term might look a little odd, as \(\beta\) is a homology class. To make this precise, we can look at it in the following way.

Let \(\beta_1, \ldots, \beta_r\) be a basis of \(H_2(X)\). For convenience, it is sometimes nice to choose it to be dual to the basis \(\gamma_1, \ldots, \gamma_r\) of \(H^2(X)\) in the sense that
\[
\int_{\beta_i}\gamma_j = \delta_i^j
\]
although this is not necessary. In such a case, we can write any \(\beta = \sum_{i=1}^r d_i\beta_i\). We then consider formal variables \(\{q_i\}_{1 \leq i \leq r}\) and define \(q^\beta\) to be
\[
q^\beta = q_1^{d_1} \cdots q_r^{d_r}.
\]
As \(q^\beta\) can be manipulated similarly (i.e. \(q^{\beta_1 + \beta_2} = q^{\beta_1}q^{\beta_2}\)), it doesn't really matter. Writing \(q^\beta\) is more invariant (i.e. does not rely on a choice of basis), which is one reason that it may be preferred. 
\end{remark}

\begin{remark}
We have made a little bit of a sleight-of-hand switch in notation which if not pointed out, is bound to be a source of confusion.

We choose our basis of cohomology \(\{\gamma_i\}\) to be a basis of the cohomology ring {\em as a vector space}, not as an algebra. As the Gromov-Witten invariants are multi-linear maps from the cohomology of \(X\) to \(\CC\), this makes sense.

Consequently, when we write
\[
\langle \gamma_0^{k_0}, \ldots, \gamma_m^{k_m}\rangle_{g,\beta}^X
\]
we are not using the exponents as {\em multiplicative} exponents, but instead as a way of denoting repeated entries. That is, we have
\[
\langle \ldots, \gamma_i^{k_i}, \ldots \rangle_{g, \beta}^X = \langle \ldots, \underbrace{\gamma_i, \ldots, \gamma_i}_{k_i}, \ldots \rangle_{g, \beta}^X
\]
\end{remark}

\begin{remark}
From a physics standpoint (and from a mirror symmetry standpoint) we should not really consider \(q\) as a formal variable at all. We should instead consider it as a coordinate on the ``K\"ahler moduli space of \(X\)'', which we denote by \(\cM_K\). That is, we can consider the function
\[
q : \cM_K \times H_2(X) \to \CC \qquad (\omega, \beta) \mapsto q^\beta = e^{2\pi i \int_\beta \omega}
\]
In this sense, we should regard the Gromov-Witten potential as a function
\[
\Phi_g^X : \cM_K \to \CC.
\]
However, we must then contend with issues of convergence. To avoid these, one can consider it to be a purely formal series; that is, we consider it as an element of the ring
\[
H^*(X) \llbracket y_0, \ldots, y_m, q\rrbracket = H^*(X) \otimes_\CC \CC \llbracket y_0, \ldots, y_m, q\rrbracket
\]
\end{remark}

So what can we do with this gadget? The first thing that we can do is to simplify it by using the divisor axiom. Let us focus, for fixed \(\beta\) and for \(1 \leq i \leq r\), on the sum
\[
\sum_{k_i} \langle\gamma_0^{k_0}, \ldots, \gamma_i^{k_i}, \ldots, \gamma_m^{k_m}\rangle_{g,\beta}^X \frac{y_i^{k_i}}{k_i!}
\]
Repeated use of the divisor axiom yields that this is
\[
\sum_{k_i} \langle\gamma_0^{k_0}, \ldots, \widehat{\gamma_i^{k_i}}, \ldots, \gamma_m^{k_m}\rangle_{g,\beta}^X \Big(\int_\beta \gamma_i\Big)^{k_i}\frac{y_i^{k_i}}{k_i!} = \langle\gamma_0^{k_0}, \ldots, \widehat{\gamma_i^{k_i}}, \ldots, \gamma_m^{k_m}\rangle_{g,\beta}^X e^{y_i\int_\beta \gamma_i}
\]
which means that the terms coming from divisors enter only within exponentials. In particular, we can write
\[
\Phi_g^X = \sum_{k_0, k_{r+1}, \ldots, k_m} \sum_{\beta \in H_2(X)} \langle\gamma_0^{k_0}, \gamma_{r+1}^{k_{r+1}}, \ldots, \gamma_m^{k_m}\rangle_{g,\beta}^X \frac{y_0^{k_0}}{k_0!} \frac{y_{r+1}^{k_{r+1}}}{k_{r+1}!} \cdots \frac{y_m^{k_m}}{k_m!} q^\beta\prod_{i=1}^r e^{y_i \int_\beta \gamma_i}
\]
There is even further simplification due to the point mapping axiom. Let us demonstrate by computing \(\Phi_0^{\PP^2}\).

\begin{example}
We will choose as a basis of cohomology the classes \(1, L, pt\), and then our homology basis will be dual to the class of a line. That is, we will choose as a generator of \(H_2(\PP^2)\) the class \(\beta\) such that \(\int_\beta L = 1\).

Let us now compute our invariants. Recall that from above, we only need to worry about non-divisor invariants. We have previously seen that if \(d \neq 0\), then \(\langle \gamma_1, \ldots, \gamma_p \rangle_{0, d\beta}^{\PP^2} = 0\) unless
\begin{enumerate}
\item \(p = 3d - 1\)
\item \(\gamma_i = pt\) for all \(i\).
\end{enumerate}

It follows that our generating function will be of the form
\[
\Phi_0^{\PP^2} = \sum_{k_0, k_1, k_2} \langle 1^{k_0}, L^{k_1}, pt^{k_2} \rangle_{0,0}^{\PP^2} \frac{y_0^{k_0}}{k_0!}\frac{y_1^{k_1}}{k_1!}\frac{y_2^{k_2}}{k_2!} + \sum_{d=1}^\infty \langle pt^{3d-1} \rangle_{0, d\beta}^{\PP^2} e^{dy_1} \frac{y_2^{3d-1}}{(3d-1)!} q^d
\]

We can now use the point mapping axiom to simplify the first term: since the homology class is zero, the terms of the first sum will all be zero unless \(k_0 + k_1 + k_2 = 3\), in which case the invariant will be nothing but the integral
\[
\int_{\PP^2} \underbrace{L \smile \cdots \smile L}_{k_1} \smile \underbrace{ pt \smile \cdots \smile pt}_{k_2}
\]
However, this is zero except in the cases
\begin{center}
\begin{tabular}{c|c|c|c}
$k_0$ & $k_1$ & $k_2$ & \(\langle 1^{k_0}, L^{k_1}, pt^{k_2} \rangle_{0,0}^{\PP^2}\) \\
\hline
2 & 0 & 1 & 1\\
1 & 2 & 0 & 1\\
\end{tabular}
\end{center}
and so the potential is
\[
\Phi_0^{\PP^2} = \frac{1}{2}(y_0^2y_2 + y_0y_1^2) + \sum_{d=1}^\infty \langle pt^{3d-1} \rangle_{0, d\beta}^{\PP^2} e^{dy_1} \frac{y_2^{3d-1}}{(3d-1)!} q^d
\]
\end{example}

We now remark that due to the point mapping axiom, that the Gromov-Witten potential contains (as its classical part) all of the triple products in the cohomology of \(X\). In a certain sense (due to Poincar\'e duality), the classical part of the potential exactly encodes the product structure on cohomology.

\begin{definition}
Let \(X\) be a smooth projective variety, and let \(\Phi_g^X\) be its Gromov-Witten potential. We define the \idx{classical part} of the genus 0 Gromov-Witten potential to be the terms with \(\beta = 0\):
\[
\Phi_{0,classical}^X(y_0, \ldots, y_m) = \sum_{\substack{k_0, \ldots, k_m\geq 0\\ k_0 + \cdots + k_m = 3}} \langle\gamma_0^{k_0}, \ldots, \gamma_m^{k_m}\rangle_{0,0}^X \frac{y_0^{k_0}}{k_0!} \cdots \frac{y_m^{k_m}}{k_m!}
\]
Note that we only look at \(\sum k_i = 3\) due to the point mapping axiom from earlier, which tells us which invariants contribute when \(\beta = 0\). Morally, this can be more simply written as
\[
\Phi_{0, classical}^X = \sum_{i,j,k} \langle \gamma_i, \gamma_j, \gamma_k\rangle_{0,0}^X y_i y_j y_k
\]
although we should be careful as this does not include necessary symmetrization factors.

We further define the \idx{quantum part} to be the other terms:
\[
\Phi_{0,quantum}^X(y_0, \ldots, y_m, q) = \sum_{k_0, \ldots, k_m} \sum_{0 \neq \beta \in H_2(X)} \langle\gamma_0^{k_0}, \ldots, \gamma_m^{k_m}\rangle_{0,\beta}^X \frac{y_0^{k_0}}{k_0!} \cdots \frac{y_m^{k_m}}{k_m!} q^\beta
\]
In particular,
\[
\Phi_0^X = \Phi_{0, classical}^X + \Phi_{0, quantum}^X
\]
\end{definition}

\begin{example}
For \(\PP^2\), we have that
\[
\Phi_{0,classical}^{\PP^2} = \frac{1}{2}(y_0^2y_2 + y_0y_1^2)
\]
\end{example}

Given a basis of \(H^*(X)\), define now the matrix \((g_{ij})\) by
\[
g_{ij} = \int_X \gamma_i \smile \gamma_j
\]
As \(X\) is smooth, Poincar\'e duality tells us that this matrix is invertible. Denote its inverse by \(g^{ij}\).

Now, the idea is the following. We have a trilinear product defined on \(H^*(X)\) given by
\[
F: \gamma_1 \otimes \gamma_2 \otimes \gamma_3 \mapsto \int_X \gamma_1 \smile \gamma_2 \smile \gamma_3
\]
This is equivalent, by a usual argument, to a map
\[
H^{k_1}(X) \otimes H^{k_2}(X) \to H^{n - k_1 - k_2}(X)^\vee
\]
However, using Poincar\'e duality, this last space is isomorphic to \(H^{k_1 + k_2}(X)\), whence the product. More precisely, if we define the elements \(\gamma^i = \sum_{k=1}^m g^{ik}\gamma_k\)
\[
\gamma_1 * \gamma_2  = \sum_{k=0}^m F(\gamma_1, \gamma_2, \gamma_k)\gamma^k
\]
then it follows that this new product is in fact nothing but the original cup product.

We now use this to define a new product, the so-called {\em (big) quantum product}; this is a deformation of the usual cup product, in a sense which we will make clear.

\begin{definition}
Let \(X\) be a smooth projective variety, let \(\{\gamma_i\}_i\) be a basis for its cohomology (as above), and let \(\Phi_0^X\) be its genus 0 Gromov-Witten potential. Define as above \(g^{ij}\) to be the inverse of the matrix \(g_{ij} = \int_X \gamma_i \smile \gamma_j\) and \(\gamma^i = \sum_{k=1}^m g^{ik}\gamma_k\).

Define the \idx{big quantum product} on \(H^*(X)\llbracket y_0, \ldots, y_m ,q \rrbracket\) to be
\[
\gamma_i * \gamma_j = \sum_{k=0}^m \frac{\partial^3 \Phi_0^X}{\partial y_i \partial y_j \partial y_k} \gamma^k
\]
\end{definition}

\begin{remark}
If we were to use only the classical part of the quantum product, this would be nothing but the usual cup product.
\end{remark}

\begin{remark}
We distinguish the \idx{small quantum product} from the big quantum product by restricting ourselves to invariants of the form \(\langle \gamma_1, \gamma_2, \gamma_3\rangle_{0, \beta}^X\), the so-called \idx{three-point invariants}.
\end{remark}

\begin{example}
Let us compute some of the quantum product for \(\PP^2\). In such a case, we choose (as usual) a basis of cohomology given by \(\gamma_0 = 1, \gamma_1 = L, \gamma_2 = pt\). It follows that
\[
\gamma^0 = \gamma_2= [pt] \qquad \qquad \gamma^1 = \gamma_1 = L \qquad \qquad \gamma^2 = \gamma_0 = 1 = [\PP^2]
\]
For simplicty, denote by 
\[
\Phi_{ijk} = \frac{\partial^3 \Phi_0^{\PP^2}}{\partial y_i \partial y_j \partial y_k}
\]
As the potential for \(\PP^2\) is given by
\[
\Phi_0^{\PP^2} = \frac{1}{2}(y_0^2y_2 + y_0y_1^2) + \sum_{d=1}^\infty N_d e^{dy_1} \frac{y_2^{3d-1}}{(3d-1)!} q^d
\]
we see that
\[
\gamma_1 * \gamma_1 = \Big(\int_{\PP^2}\gamma_1 \smile \gamma_1\Big)[pt] + \Phi_{111} L + \Phi_{112} [\PP^2]
\]
and in particular that the product is not of pure degree. Moreover, it contains (as we expect) a term corresponding to the original product, as well as other non-classical terms.
\end{example}

Now, a natural question that should arise whenever we define a new product on some algebra is what properties it has. Is it commutative? Associative?

We begin with the following fact.

\begin{theorem}
The genus 0 Gromov-Witten potential \(\Phi\) of a smooth projective variety \(X\) satisfies the WDVV equation
\[
\sum_{a,b} \Phi_{ija}g^{ab}\Phi_{bk\ell} = (-1)^{\deg \gamma_i(\deg \gamma_j + \deg \gamma_k)} \sum_{a,b} \Phi_{jka}g^{ab}\Phi_{bi\ell}
\]
for all \(0 \leq i, j, k, \ell \leq m\), where
\[
\Phi_{ijk} = \frac{\partial^3 \Phi}{\partial y_i \partial y_j \partial y_k}
\]
\end{theorem}

We will not go over a proof of this, but this essentially relies on the facts that
\begin{enumerate}
\item there is a forgetful map \(\cMb_{0,n}(X, \beta) \to \cMb_{0,4}\), and
\item all divisors on \(\cMb_{0, 4}\) are linearly equivalent.
\end{enumerate}

We can now conclude the following.

\begin{theorem}
The big quantum product \(*\) defined on \(H^*(X)\llbracket y_0, \ldots, y_m, q \rrbracket\) is associative and graded commutative.
\end{theorem}

We are now in a position to show our main result; that is, we will use all of the above formalism to compute the recursion for the number of degree \(d\) plane curves passing through \(3d -1 \) points.

\begin{theorem}
Let \(N_d\) denote the number of degree \(d\) plane curves passing through \(3d -1\) points. Then \(N_d\) satisfies the recurrence relation
\[
N_d = \sum_{d_1 + d_2 = d}\Big(d_1^2d_2^2 N_{d_1}N_{d_2} {3d - 4 \choose 3d_1 - 2} - d_1^3d_2N_{d_1}N_{d_2} {3d - 4 \choose 3d_1 - 1}\Big)
\]
with initial conditions \(N_1 = 1\).
\end{theorem}

\begin{proof}
The WDVV equation for \(\PP^2\) is given by
\[
\Phi_{222} = \Phi_{112}^2 - \Phi_{111}\Phi_{122}.
\]
Note that the term of interest, \(N_d\), shows up on the left-hand side in the term
\[
N_d e^{dy_1} \frac{y_2^{3d-4}}{(3d-4)!}
\]
and so this suggests looking for terms on the right-hand side whose exponent of \(y_2\) is also \(3d-4\).

The first term, \(\Phi_{112}^2\) is given by
\[
\Big(\sum_{d=1}^\infty d^2N_d e^{dy_1} \frac{y_2^{3d-2}}{(3d-2)!}q^d\Big)^2
\]
and so out desired terms come from picking all \(d_1 + d_2 = d\) giving us
\[
\sum_{d_1 + d_2 = d} d_1^2d_2^2 N_{d_1}N_{d_2} e^{dy_1} \frac{y_2^{3d-4}}{(3d_1 - 2)!(3d_2-2)!}
\]
Similarly, in the second term we have
\[
\Big(\sum_{d=1}^\infty d^3N_d e^{dy_1} \frac{y_2^{3d-1}}{(3d-1)!}q^d\Big)
\Big(\sum_{d=1}^\infty dN_d e^{dy_1} \frac{y_2^{3d-3}}{(3d-3)!}q^d\Big)
\]
which yields a term of the form
\[
\sum_{d_1 + d_2 = d} d_1^3d_2 N_{d_1}N_{d_2} e^{dy_1} \frac{y_2^{3d-4}}{(3d_1 - 3)!(3d_2-1)!}
\]
The formula now follows from equating the left- and right-hand sides.
\end{proof}

\section{Conclusion}

Counting curves in varieties is hard. Computing Gromov-Witten invariants is also quite hard. Nevertheless, the formalism so-obtained is quite powerful in that it not only introduces a rigorous definition of counts of curves (modulo some details), but it also provides a lot of structure that one can use to understand these counts.

It seems in general that the key to understanding how to solve difficult problems is often to make them seemingly harder---we find an infinite family of similar problems, but use then the relations between each of the problems to help solve them all in one fell swoop.

Computing the number of degree 5, or 7, or 83,124 rational plane curves would have been an insurmountable problem before. Gromov-Witten theory, however, lets us see the underlying pattern behind these numbers, and to solve them all in one fell swoop.

\section{Exercises}

\begin{enumerate}
\item\label{ex_genus_plane_curve} Show that a smooth degree \(d\) plane curve has genus \(\frac{(d - 1)(d - 2)}{2}\). Hint: Consider the adjunction formula, that says that for any smooth divisor \(Y \subset X\), that
\[
K_Y \cong K_X|_Y \otimes N_{Y/X}
\]
What is the relationship between the degree of the canonical bundle of a curve and its genus?

\item Show that a marked nodal curve \((C, x_1, \ldots, x_n)\) has only finitely many automorphisms whenever
\[
2g - 2 + n > 0.
\]

\item Go through a number of papers and try to see which ones use \verb|\langle|, \verb|\rangle| and which ones use \(<, >\). Which look better?

\item Show that the forgetful map
\[
\pi_{n+1} : \cMb_{g,n+1}(X, \beta) \to \cMb_{g,n}(X, \beta)
\]
exists provided that one of
\begin{enumerate}
\item \(n + 2g \geq 4\)
\item \(\beta \neq 0\) and \(n \geq 1\)
\end{enumerate}
is satisfied.

\item With a bit of fudging, we will make an effort to compute the number of lines through a pair of points by computing (by hand) the Gromov-Witten invariant \(\langle pt, pt \rangle_{0, 1}^{\PP^2} = 1\).

As this will be an integral over the moduli space \(\cMb_{0, 2}(\PP^2, H)\), we need to understand this moduli space. We first note that as before, we have that
\[
\cMb_{0,0}(\PP^2, H) = (\PP^2)^*
\]
i.e. it is the collection of lines in \(\PP^2\).
\begin{enumerate}
\item Show (loosely) that the moduli space \(\cMb_{0, n+1}(X, \beta)\) is the \idx{universal family} over \(\cMb_{0, n}(X, \beta)\). That is, it fits into a digram
\[
\xymatrix{
\cMb_{0,n+1}(X, \beta) \ar[d]_\pi \ar[rr]^f & & X \\
\cMb_{0, n}(X, \beta)
}
\]
where the fibre over a point in the base (i.e. a stable map \(f : (C, x_1, \ldots, x_n) \to X\)) is the curve together with the map \(f\). Hint: This may be easier if you think of the similar case of the moduli space of curves, where it is simpler to see that \(\cMb_{0, n+1} \to \cMb_{0, n}\) is the universal family.

\item Using the previous part, describe the moduli space \(\cMb_{0, 2}(\PP^2, H)\) as
\[
\cMb_{0, 2}(\PP^2, H) = \{ (\ell, x, y) \in (\PP^2)^* \times \PP^2 \times \PP^2 \mid x, y \in \ell\}
\]
with the evaluation maps \(ev_i\) being the projections onto the two copies of \(\PP^2\).

\item Using the fact (and this is the loosest part of the exercise) that
\[
ev_1^*([pt]) = [\{(\ell, x, y) \mid x = pt\}]
\]
(and similarly for \(ev_2\)), compute the Gromov-Witten invariant \(\langle pt, pt \rangle_{0, H}^{\PP^2}\).
\end{enumerate}

\item\label{ex_lattices} Show that the number of index \(d\) sublattices of \(\ZZ^2\) is given by
\[
\sigma_1(d) = \sum_{k \mid d} k
\]
and hence that the Gromov-Witten invariant \(\langle pt \rangle_1^E = \sigma_1(d)\).

\item How would the computation from the previous exercise change if we were to look at the invariant \(\langle\ \rangle_1^E\)?

\item Define \(F(\gamma_1, \gamma_2, \gamma_3) = \int_X \gamma_1 \smile \gamma_2 \smile \gamma_3\). Verify that the prouct \(*\) given by
\[
\gamma_1 * \gamma_2 = \sum_{k=1}^m F(\gamma_1, \gamma_2, \gamma_k)\gamma^k
\]
is the usual cup product for the spaces
\begin{enumerate}
\item \(\PP^2\)
\item \(S^1 \times S^1\)
\item A curve of genus \(g > 1\).
\end{enumerate}
Note that a good choice of cohomology basis may make this much easier.

\item Define the \idx{small quantum product} via
\[
\gamma_i *_s \gamma_j = \sum_{k=0}^m \sum_{\beta \in H_2(X)} \langle \gamma_i,\gamma_j,\gamma_k\rangle_{0,\beta}^X q^\beta \gamma^k
\]
and compute the small quantum product for \(\PP^N\). That is, compute the product structure that is obtained on the ring \(H^*(\PP^N) \otimes_\CC \CC[q]\). Hint: using associativity of the product, show that it is enough to compute \(H^n * H\), where \(H\) is the hyperplane class. This can then be computed quite simply based on what we have seen elsewhere.

\item Verify that the WDVV equations for \(\PP^2\) are given by
\[
\Phi_{222} + \Phi_{111}\Phi_{122} = \Phi_{112}^2
\]

\item Compute the first few terms \(N_d\) (for \(d = 2, 3, 4, 5\)) and verify that they agree with the predictions
\[
N_2 = 1 \qquad \qquad N_3 = 12 \qquad\qquad N_4 = 620 \qquad\qquad N_5 = 87,304
\]
\end{enumerate}

\bibliographystyle{amsplain}
\bibliography{notes}

\providecommand{\bysame}{\leavevmode\hbox to3em{\hrulefill}\thinspace}
\providecommand{\MR}{\relax\ifhmode\unskip\space\fi MR }
\providecommand{\MRhref}[2]{%
  \href{http://www.ams.org/mathscinet-getitem?mr=#1}{#2}
}
\providecommand{\href}[2]{#2}
\begin{thebibliography}{1}

\bibitem{agv}
Dan Abramovich, Tom Graber, and Angelo Vistoli, \emph{Algebraic orbifold
  quantum products}, Orbifolds in mathematics and physics ({M}adison, {WI},
  2001), Contemp. Math., vol. 310, Amer. Math. Soc., Providence, RI, 2002,
  pp.~1--24. \MR{1950940 (2004c:14104)}

\bibitem{agv_long}
\bysame, \emph{Gromov-{W}itten theory of {D}eligne-{M}umford stacks}, Amer. J.
  Math. \textbf{130} (2008), no.~5, 1337--1398. \MR{2450211 (2009k:14108)}

\bibitem{bf_normal_cone}
K.~Behrend and B.~Fantechi, \emph{The intrinsic normal cone}, Invent. Math.
  \textbf{128} (1997), no.~1, 45--88. \MR{1437495 (98e:14022)}

\bibitem{coxetal}
D.~Cox and S~Katz, \emph{Mirror symmetry and algebraic geometry}, ch.~7,
  American Mathematical Society, 1999.

\bibitem{gillam}
William Gillam, \emph{Hyperelliptic gromov-witten theory}, Ph.D. thesis,
  Columbia University, 2008.

\bibitem{hirz-book}
Friedrich Hirzebruch, \emph{Topological methods in algebraic geometry},
  Springer, Berlin New York, 1995.

\bibitem{knud}
Finn~F. Knudsen, \emph{The projectivity of the moduli space of stable curves.
  {II}. {T}he stacks {$M_{g,n}$}}, Math. Scand. \textbf{52} (1983), no.~2,
  161--199. \MR{702953 (85d:14038a)}

\bibitem{op_gw_hurwitz}
A.~Okounkov and R.~Pandharipande, \emph{Gromov-witten theory, hurwitz theory,
  and completed cycles}, Anal. Math. \textbf{163} (2006), no.~2, 517--560.

\bibitem{pt_ways}
R~Pandharipande and R.~Thomas, \emph{13/2 ways of counting curves}, Proceeding
  of School on Moduli Spaces, 2011.

\end{thebibliography}


\end{document}